\documentclass[11pt]{article}
\usepackage{amsmath,amssymb,amsthm,graphicx,epsfig,subfigure,float,url}
\usepackage[colorlinks=true,citecolor={ForestGreen},linkcolor={Periwinkle}]{hyperref}
\usepackage[left=2cm,right=2cm,top=1.5cm,bottom=1.5cm]{geometry}
\usepackage{pdfsync}
\usepackage{color}
\usepackage{dsfont}
\usepackage{verbatim}
\usepackage[applemac]{inputenc}
\usepackage{cleveref}
\usepackage{mathrsfs}
\usepackage[usenames,dvipsnames]{pstricks}
\usepackage{pst-grad} 
\usepackage{pst-plot} 
\usepackage{stmaryrd}

\def\R{\textrm{I\kern-0.21emR}}
\def\N{\textrm{I\kern-0.21emN}}
\def\O{{\Omega}}

\def\Re{\operatorname{Re}}
\newcommand{\C} {\mathbb{C}}

\renewcommand{\geq}{\geqslant}
\renewcommand{\leq}{\leqslant}

\newtheorem{theorem}{Theorem}

\newtheorem{lemma}{Lemma}
\theoremstyle{definition}
\theoremstyle{definition}\newtheorem{remark}{Remark}

\newcommand{\Hun}{\mathbf{(H_1)}}
\newcommand{\Hdeux}{\mathbf{(H_{\boldsymbol{L}})}}
\newcommand{\Htrois}{\mathbf{(H_2)}}

\title{Large-time optimal observation domain for linear parabolic systems}

\author{Idriss Mazari-Fouquer\footnote{CEREMADE, UMR CNRS 7534, Universit\'e Paris-Dauphine, Universit\'e PSL, Place du Mar\'echal De Lattre De Tassigny, 75775 Paris cedex 16, France (\texttt{mazari@ceremade.dauphine.fr}).}
\and
Yannick Privat\footnote{Universit\'e de Lorraine, CNRS, Inria, IECL, F-54000 Nancy, France (\texttt{yannick.privat@univ-lorraine.fr}).}~\footnote{Institut Universitaire de France (IUF).}
	\and Emmanuel Tr\'elat\footnote{Sorbonne Universit\'e, CNRS, Universit\'e Paris Cit\'e, Inria, Laboratoire Jacques-Louis Lions (LJLL), F-75005 Paris, France (\texttt{emmanuel.trelat@sorbonne-universite.fr}).} 
}

\date{}

\begin{document}

\maketitle

\begin{abstract}
Given a well-posed linear evolution system settled on a domain $\Omega$ of $\R^d$, an observation subset $\omega\subset\Omega$ and a time horizon $T$, the observability constant is defined as the largest possible nonnegative constant such that the observability inequality holds for the pair $(\omega,T)$.
In this article we investigate the large-time behavior of the observation domain that maximizes the observability constant over all possible measurable subsets of a given Lebesgue measure. We prove that it converges exponentially, as the time horizon goes to infinity, to a limit set that we characterize.
The mathematical technique is new and relies on a quantitative version of the bathtub principle.
\end{abstract}

\noindent\textbf{Keywords:} parabolic equations, observability, shape optimization, quantitative bathtub principle.

\medskip

\noindent\textbf{AMS classification:} 93B07, 35L05, 49K20, 42B37.

\section{Introduction}\label{secintro}

Let $d\in\N^*$ and let $\Omega$  be an open bounded smooth connected subset of $\R^d$.
Let $q\in\N^*$. We consider the evolution system
\begin{equation}\label{heatEq}
\partial_t y+A_{0} y=0
\end{equation}
where $-A_{0}:D(A_0)\rightarrow L^2(\Omega,\C^q)$ is a densely defined operator generating a strongly continuous semigroup on $L^2(\Omega,\C^q)$.
Given any $y^0\in D(A_{0})$, there exists a unique solution $y\in C^0([0,+\infty),D(A_{0}))\cap C^1((0,+\infty),L^2(\Omega,\C^q))$ of \eqref{heatEq} such that $y(0,\cdot)=y^0(\cdot)$.

For an evolution system like \eqref{heatEq}, it is often required in practice, for instance in engineering problems, to reconstruct initial or final data of solutions, based on partial measurements performed on a subset $\omega$ of $\Omega$ over a a time horizon $T$. This inverse problem is feasible as soon as the system is observable on $\omega$ in time $T$, which is mathematically modelled by an \textit{observability inequality}, as follows.
For any measurable subset $\omega$ of $\Omega$ and any $T>0$, the system \eqref{heatEq} is said to be \textit{observable} on $\omega$ in time $T$ if there exists $C>0$ such that
\begin{equation}\label{ineqobs}
C \Vert y(T,\cdot)\Vert_{L^2(\Omega,\C^q)}^2
\leq \int_0^T\!\!\!\int_\omega \vert y(t,x)\vert^2 \,dx \, dt
\end{equation}
for every solution of \eqref{heatEq} such that $y(0,\cdot)\in D(A_{0})$. 
In \eqref{ineqobs}, $\vert y\vert$ is the Hermitian norm of $y\in\C^q$.

Many results have been established in the literature regarding observability properties for some classes of partial differential equations. Since, in this article, we are going to focus on the case of \textit{parabolic} equations (see Section \ref{sec:mainHyp} for assumptions on $A_0$), here we simply quote few of them. For the heat equation, that is, when $-A_0$ is the Dirichlet Laplacian, observability holds true for any $T>0$ and any subset of positive Lebesgue measure (see \cite{Fursikov, LebeauRobbiano} for open subsets and see \cite{apraiz2014observability} for the extension to measurable subsets).
The same holds true for anomalous diffusion equations, that is, when $A_0$ is some power $\alpha>1/2$ of the Dirichlet Laplacian (see \cite{Miller}), and for the Stokes equation (see \cite[Lemma 1]{Puel} and \cite{CanZhang}). We also refer to \cite[Chapter 9]{TucsnakWeiss} for some general results establishing observability for parabolic systems.

The \emph{observability constant} is defined as the largest possible nonnegative constant for which the observability inequality \eqref{ineqobs} holds, i.e., by 
\begin{equation}\label{defCT}
C_T(\mathds{1}_\omega)=\inf\left\{ \frac{ \int_0^T\!\!\int_\Omega \mathds{1}_\omega(x)|y(t,x)|^2\,dx \, dt }{\Vert y(T,\cdot)\Vert^2_{L^2(\Omega,\C^q)}} \ \big\vert\  y^0\in D(A_{0}) \setminus\{0\} \right\},
\end{equation}
where $\mathds{1}_\omega$ denotes the characteristic function of $\omega$. Note that \eqref{heatEq} is observable on $\omega$ in time $T$ if and only if $C_T(\mathds{1}_\omega)>0$. 
The observability constant defined by \eqref{defCT} provides an account, in some sense, for the well-posedness of the above-mentioned inverse problem: the larger the constant, the more favorable it is to solve the inverse problem. This is why it is important, in practice, to try to choose the observation subset $\omega$ in an optimal way and, when it is possible, the horizon of time $T$ over which observations are performed.
Of course, in practice, there are limitations on the choice of the subset $\omega$, starting with its Lebesgue measure. This is why, in what follows, we will consider measurable subsets of a given measure.

Let $L\in(0,1)$ be fixed. Given any $T>0$, we consider the optimal design problem
\begin{equation}\label{maxCT}
\boxed{\delta_T = \sup_{\mathds{1}_\omega\in \mathcal{U}_L} C_T(\mathds{1}_\omega)}
\end{equation}
consisting of maximizing the observability constant over the set
\begin{equation}\label{defUL}
\mathcal{U}_L = \{ \mathds{1}_\omega \in L^\infty(\Omega;\{0,1\}) \ \vert\ \omega\ \textrm{is a measurable subset of}\ \Omega \ \textrm{of Lebesgue measure}\ \vert\omega\vert=L\vert\Omega\vert\}.
\end{equation}
This shape optimization problem models the best shape and placement of sensors for the evolution system \eqref{heatEq} in time $T$. If an optimal set $\omega^\star_T$ exists, then it represents the best possible place to install some (adequately shaped) sensors. 

Unfortunately, the shape optimization problem \eqref{defUL} is very difficult to handle, not only because of the complexity of the definition of the observability constant \eqref{defCT}, which is an infimum over all possible solutions of \eqref{heatEq}, but also because of the set of unknowns $\mathcal{U}_L$, which is very large and does not enjoy good compactness properties. Actually, it is not even known if a maximizer $\omega^\star_T$ exists for \eqref{maxCT}. 

To remedy this lack of compactness, some constraints may be added. The search of an optimal set may be restricted to a compact finite-dimensional set, like in \cite{Seinfeld, Morris, MR2743850, MR1832938}, or by adding constraints on the BV norm, on the number of connected components, or anything yielding compactness (see, e.g., \cite{Henrot-Pierre}), or by fixing the initial data, like in \cite{FernandezCaraMunch, MunchPeriago, MR3393270}. However, we want to keep the maximal freedom on the choice of the observation subset and this is why we consider the general set $\mathcal{U}_L$ defined by \eqref{defUL}.

This choice comes at a price: the set $\mathcal{U}_L$ is not even closed for the weak-star topology of $L^\infty(\Omega)$, thus making even the basic question of existence of an optimizer challenging.
Such a question is well known in shape optimization, and this issue is often tackled by considering a relaxed version of the problem.
Here, the relaxation procedure consists of extending the definition of $C_T$ to the $L^\infty$ weak-star closure of $\mathcal{U}_L$, which is the set
\begin{equation*}
\overline{\mathcal{U}}_L = \left\{ a\in L^\infty(\Omega;[0,1])\ \vert\ \int_{\Omega} a(x)\,dx=L\vert\Omega\vert\right\}
\end{equation*}
and to define, in accordance with \eqref{defCT},
\begin{equation}\label{defCTa}
C_T(a)= \inf\left\{ \frac{ \int_0^T\!\int_\Omega a(x)|y(t,x)|^2\,dx \, dt }{\Vert y(T,\cdot)\Vert^2_{L^2(\Omega,\C^q)}} \ \big\vert\  y^0\in D(A_{0}) \setminus\{0\} \right\} \qquad\forall a\in \overline{\mathcal{U}}_L.
\end{equation}
Then, we consider the relaxed shape optimization problem
\begin{equation}\label{maxCTa}\tag{$\mathscr{P}_T$}
\boxed{\bar\delta_T=\sup_{a\in \overline{\mathcal{U}}_L}C_T(a)}
\end{equation}
Now, since $\overline{\mathcal{U}}_L$ is compact for the $L^\infty$ weak-star topology (see \cite[Prop.~7.2.17]{Henrot-Pierre}) and $C_T$ is upper semi-continuous (as an infimum of bounded linear functionals), it follows that Problem~\eqref{maxCTa} has at least one solution $a_T^\star$.
Characterizing $a_T^\star$ is more intricate. In this article, we will study $a_T^\star$ as $T \to \infty$ and also give some results on its small-time asymptotics as $T\rightarrow 0$. 

Of course, we have $\delta_T \leq \bar\delta_T=C_T(a_T^\star)$ but we do not know, in general, whether the inequality may be strict or not. Anyhow, observe the following \emph{a priori} counter-intuitive result: in some cases, the constant density $a\equiv L$ is not a maximizer (see \cite[Prop. 2]{MR3048589}).

To tackle the difficulty of optimizing the observability constant \eqref{defCT}, or its relaxed version \eqref{defCTa}, the point of view adopted in the series of papers
\cite{MR3132418, MR3048589, 
MR3325779, MR3500831, MR3502963, MR3632257, MR3925555}
consisted in considering a \emph{randomized} version $C_{T,{\rm rand}}(a)$ of the observability constant, defined by taking the infimum in \eqref{defCTa} not on all but on \emph{almost all} initial data in an appropriate sense, thus yielding a more tractable expression of the functional, namely, 
\begin{equation}\label{defCTranda}
C_{T,{\rm rand}}(a) = \inf_{j\in\N^*} \gamma_j(T)\int_\Omega a(x)\vert\phi_j(x)\vert^2\, dx ,
\end{equation}
where $(\phi_j)_{j\in\N^*}$ is a Hilbert basis of eigenfunctions of $A_0$ and $(\gamma_j(T))_{j\in\N^*}$ is a sequence of positive coefficients depending on the operator. 
We have $C_{T,{\rm rand}}(a)\geq C_{T}(a)$ for every $a\in \overline{\mathcal{U}}_L$ and the inequality may be strict  \cite[Theorem 1]{MR3325779}. The randomized observability constant $C_{T,{\rm rand}}(a)$ is in some sense ``less pessimistic" than the deterministic observability constant $C_{T}(a)$ which provides an account for the worst observation cases.
From the harmonic analysis point of view, the constant $C_{T,{\rm rand}}(a)$ reflects the independent observations on each mode, with no interaction between them, while the constant $C_{T}(a)$ also takes into account all interactions between modes, through crossed terms in a spectral expansion of solutions.

In the above-mentioned series of papers, the problem of maximizing the randomized (relaxed or not) observability constant has been solved for various classes of evolution equations, showing strong differences between parabolic and hyperbolic cases. The results have revealed a close relationship with asymptotic properties of high frequency eigenfunctions (quantum ergodicity properties). 
Roughly speaking, it has been shown that, in hyperbolic cases, all modes count and a maximizer does not exist in general, while in parabolic cases, there exists a unique maximizing domain which is moreover characterized and computable from a finite number of modes only.

The question of maximizing the deterministic ({\emph{i.e.} without randomization}) observability constant $C_T(a)$ has remained unsolved theoretically, although it has been investigated numerically in 1D (see \cite{bottois2021optimization, munch2009optimal}). 

Although we do not solve this difficult problem, the present article takes us one step further by studying the large-time behavior of the optimal observability constant for parabolic evolution systems and of the associated maximizers. 
It is interesting to note that the authors of \cite{AllaireMunchPeriago} studied the relaxed problem for two-phase heat equations and proved that the relaxed maximizers converge, as $T\rightarrow+\infty$, to relaxed maximizers of the corresponding stationary problem.

\smallskip

In Section \ref{mainsec:mainresProof}, we state our main results concerning the large-time asymptotics of the relaxed maximizers. 
We prove in Theorems \ref{theo:largeT} and \ref{theo:sigma1} that, in large time $T$, maximizing $C_T(a)$ over $\overline{\mathcal{U}}_L$ is approximately equivalent to maximizing the lowest eigenvalue $\sigma_1(a)$ of a matrix -- when $A_0$ is selfadjoint and its lowest eigenvalue $\lambda_1$ is simple then $\sigma_1(a)=\int_\Omega a\vert\phi_1\vert^2$ where $\phi_1$ is a normalized eigenfunction of $-A_0$ associated to $\lambda_1$. The latter limit problem has a unique solution $a_1=\mathds{1}_{\omega^*}$, with $\omega^*$ enjoying nice regularity properties. We prove that any maximizer $a^\star_T\in\overline{\mathcal{U}}_L$ of $C_T$ converges to $\mathds{1}_{\omega^*}$ in $L^1$ norm, exponentially as $T\rightarrow+\infty$ at a rate given by the spectral gap of the operator $A_0$.
To determine the speed of convergence, we use an improved, quantitative version of the so-called {\it bathtub principle}, inspired by  \cite{MRBSIAM} and \cite[Prop. 2.7]{casas-Wachsmuth}. The use of this technique is one of the main novelties of the present paper.

We provide in Section \ref{sec:comments} several examples and we comment on the assumptions under which our main results are established. 
Following a relevant comment by Sylvain Ervedoza during a presentation of one of us, we comment in Section \ref{sec_comments_proof} on the strategy of proof that we develop, showing that, although the main result is perfectly intuitive, a naive approach of proof is bound to fail, and we have to resort to a more elaborate technique using kinds of quasi-maximizers.

The proofs of the main results are done in Section \ref{sec_proof}. 

In the last Section \ref{sec:concl}, we comment on related problems: first, the optimal controllability domain problem, which is dual to the one investigated in this paper; second, in contrast to what is developed here, the small-time asymptotics of the optimal observability problems. Finally we give some open problems.

\section{Large-time behavior of maximizers}\label{mainsec:mainresProof}
\subsection{Main results}\label{sec:mainHyp}
Throughout the paper, given $y=(y_1,\ldots,y_q)$ and $z=(z_1,\ldots,z_q)$ in $\C^q$, we denote by $\vert y\vert=(\sum_{i=1}^q\vert y_i\vert^2)^{1/2}$ the Hermitian norm and by $y\cdot\bar z=\sum_{i=1}^q y_i\bar z_i$ the Hermitian inner product. 

We consider the following assumptions:
\begin{itemize}
\item[$\Hun$] There exists a Hilbert basis $(\phi_j)_{j\in\N^*}$ of $L^2(\Omega,\C^q)$ consisting of (complex-valued) eigenfunctions of $A_{0}$, associated with eigenvalues $(\lambda_j)_{j\in\N^*}$ of finite multiplicity such that 
\[
\Re(\lambda_1)\leq\cdots\leq \Re(\lambda_j)\leq\cdots\to +\infty \quad \text{as }j\to +\infty.
\]
\item[$\Htrois$] For every $j\in\N^*$, the eigenfunction $\phi_j$ is analytic in $\Omega$.

\medskip

We define
$$
J_1=\{j\in \N^*\mid \Re \lambda_j=\Re \lambda_1\} ,
$$
$$
p_0=\min\{p\in \N^*\mid \Re \lambda_p>\Re \lambda_1\}=\min (\N^*\backslash J_1)\geq 2.
$$
The positive real number $\Re \lambda_{p_0}-\Re \lambda_1$ is the {\it spectral gap} of the operator $A_0$.

\medskip

\item[$\Hdeux$] (\textit{Bathtub property at level $L$}) For every $j\in J_1$, the function $\phi_j$ belongs to $C^1(\overline{\Omega})$. Given any real-valued function $\Phi$ on $\Omega$ written as
$$
\Phi=\sum_{p\in J_1}\eta_p\bigg|\sum_{j\in J_1}\beta_{j,p}\phi_j\bigg|^2
$$
for some $(\beta_{j,p})_{j,p\in J_1}\in \R^{N_1\times P}\backslash \{0\}$, $\{\eta_p\}_{p\in J_1}\in [0,1]^P$ such that $\sum_{p\in J_1}\eta_p=1$, the set $\omega=\{\Phi\geq \nu\}$ where $\nu>0$ is such that $|\omega|=L |\Omega|$, satisfies 
$$
\inf_{\partial\omega\backslash \partial\Omega}|\nabla  \Phi|> 0.
$$

\end{itemize}

In Section~\ref{sec:comments}, we will discuss these assumptions, give various examples, and show that Assumption~$\Hdeux$ cannot be weakened.

This is because of Assumptions $\Hun$ and $\Htrois$ that we have used, in this paper and in particular in its title, the word ``parabolic". However, the word may be a bit of a misnomer because it often means that the semigroup generated by $-A_0$ is analytic. However, there exist some analytic hypoelliptic operators satisfying all above assumptions, while the semigroup they generate is not analytic \cite[Remark 4.1]{CdVHT}.

\medskip

Under $\Hun$, for any  $y_0\in D(A_0)$, the solution $y$ of \eqref{heatEq} such that $y(0,\cdot)=y_0$ can be expanded as
\begin{equation*}
y(t,x)=\sum_{j=1}^{+\infty}a_je^{-\lambda_jt}\phi_j (x),
\end{equation*}
where 
\begin{equation*}
a_j = \int_\Omega y^0(x) \overline{\phi_j} (x)\, dx \qquad\forall j\in \N^*.
\end{equation*}
Using the change of variable $b_j=a_je^{-\lambda_jT}$ and a homogeneity argument, we have
\begin{equation}\label{defCTdeterministic}
C_T(a)=\inf_{\sum_{j=1}^{+\infty}|b_j|^2=1}\ \  \int_0^T\!\! \int_\Omega a(x)\bigg|\sum_{j=1}^{+\infty}b_je^{\lambda_j t}\phi_j(x)  \bigg|^2 dx \, dt.
\end{equation}

Before stating our main results, let us give some notations.
 For any $j\in \N^*$, set
\begin{equation*}
\gamma_j(T)=
\left\{ \begin{array}{lll}
\displaystyle\frac{e^{2\Re (\lambda_j)T}-1}{2\Re(\lambda_j)} & \textrm{if} & \Re(\lambda_j)\neq 0, \\
T & \textrm{if} & \Re(\lambda_j)=0 .
\end{array}\right.
\end{equation*}
We define the Hermitian matrix 
$$
M_1(a)=\left(\int_\Omega a (x)\,\phi_i(x)\cdot\overline{\phi_j(x)}\, ds\right)_{i,j\in J_1}
$$
and we denote by $\sigma_1(a)$ the lowest eigenvalue of $M_1(a)$.
When $\# J_1=1$, i.e., $J_1=\{1\}$ (this is the case if $A_0$ is selfadjoint and $\lambda_1$ is simple), then $\sigma_1(a) = M_1(a)=\int_\Omega a\vert\phi_1\vert^2$.

\begin{theorem}\label{theo:largeT}
Under $\Hun$, we have
\begin{equation}\label{Eq:HP5}
 \bar\delta_T \sim \gamma_1(T)\max_{\mathds{1}_\omega\in\mathcal{U}_L}\sigma_1(\mathds{1}_\omega)
\quad \textrm{as}\quad T\to +\infty .
\end{equation}
Under $\Hun$, $\Htrois$ and $\Hdeux$, any solution $a^\star_T\in\overline{\mathcal{U}}_L$ of Problem~\eqref{maxCTa} satisfies
\begin{equation}\label{Eq:HP6}
\Big\Vert a_T^\star - \underset{a\in \overline{\mathcal{U}}_L}{\operatorname{argmax}}~\sigma_1(a)\Big\Vert_{L^1(\Omega)} = \operatorname{O}\left(e^{-(1-\xi)\Re(\lambda_{p_0}-\lambda_1)T/2}\right)\qquad \forall \xi>0 .
\end{equation}
\end{theorem}
\begin{remark}
Note that \eqref{Eq:HP5} makes sense since
$$
\max_{\mathds{1}_\omega\in\mathcal{U}_L}\sigma_1(\mathds{1}_\omega)=\max_{a\in\overline{\mathcal{U}}_L}\sigma_1(a)\geq \sigma_1(L)=L>0.
$$
\end{remark}
Theorem~\ref{theo:largeT} highlights the interest of the shape optimization problem
\begin{equation}\label{max:sigma1}\tag{$\mathscr{P}_{\sigma_1}$}
\max_{a\in\overline{\mathcal{U}}_L}\sigma_1(a).
\end{equation}

\begin{theorem}\label{theo:sigma1}
Under $\Hun$ and $\Hdeux$, the problem \eqref{max:sigma1} has a unique\footnote{Here and in the sequel, it is understood that the optimal set is unique within the class of all measurable subsets of $\Omega$ quotiented by the set of all measurable subsets of $\Omega$ of zero measure.} solution $a_1=\mathds{1}_{\omega^*}\in\mathcal{U}_L$ satisfying the following properties:
\begin{itemize}
\item There exists $\mu^*>0$ such that 
\begin{equation}\label{carac:omegastar}
\omega^*=\bigg\{x\in \Omega \ \mid \ \sum_{k\in J_1}\alpha_k\bigg|\sum_{j\in J_1}b_j^k\phi_j(x)\bigg|^2 > \mu^*\bigg\},
\end{equation}
for some $(\alpha_k)_{k\in J_1}\in [0,1]^{\# J_1}$ and $(b_j^k)_{j\in J_1}\in\C^{\# J_1}$ such that 
$$
\sum_{j\in J_1}\alpha_j= \sum_{j\in J_1}|b_j^k|^2=1\qquad \forall k\in J_1.
$$
\item There exists $K>0$ such that
\begin{equation}\label{Eq:EstQuant}
\sigma_1(a_1)\geq \sigma_1(a) + K\Vert a-a_1\Vert_{L^1(\O)}^2\qquad \forall a \in \overline{\mathcal U}_L.
\end{equation}
\item If in addition $\Htrois$ holds, then $\omega^*$ is semi-analytic\footnote{A subset $\omega$ of a real analytic finite dimensional manifold $M$ is said to be semi-analytic if it can be written in terms of equalities and inequalities of analytic functions.
Recall that semi-analytic subsets are Whitney stratifiable (see \cite{Hardt,Hironaka}) and enjoy local finiteness properties, such that: local finite perimeter, locally finite number of connected components, etc.}. In particular, it has a finite number of connected components.
\end{itemize}
\end{theorem}

\begin{remark}
Concerning the randomized observability constant defined by \eqref{defCTranda}, it follows from some results in \cite{MR3325779} that, for $T$ is large enough,
$$
\sup_{a\in \overline{\mathcal{U}}_L}C_{T,{\rm rand}}(a)=\gamma_1(T)\max_{\mathds{1}_\omega\in\mathcal{U}_L}\sigma_1(\mathds{1}_\omega).
$$
Therefore, Theorem \ref{theo:largeT} shows that, in large time, maximizing the deterministic observability constant is almost equivalent to maximizing the randomized observability constant.
\end{remark}

\begin{remark}
The exponential rate of convergence~\eqref{Eq:HP6} in Theorem~\ref{theo:largeT} is established thanks to a \emph{quantitative} version of the \emph{bathtub principle}. The bathtub principle is an elementary inequality yielding the solutions of the shape optimization problem
$$
\sup_{\omega \in \mathcal{U}_L}\int_\omega f
$$
where $f$ is a given real-valued integrable function on $\Omega$. In recent years, quantified versions of this principle were derived and used in optimal control problems. We will build on the version stated in \cite[Proposition 24]{MRBSIAM}. We refer to Section \ref{sec:biblio} for more background on this inequality but we merely state here that, while it is now a known tool for optimal control theory, it is to the best of our knowledge the first time that such quantified inequalities are used in the study of observability constants.
\end{remark}

\begin{remark}
 In the theory of rearrangements, the quantitative estimate \eqref{Eq:EstQuant} can be seen as a strengthened Hardy-Littlewood inequality; to the best of our knowledge, the first paper to investigate such strengthening was  \cite{CianchiFerone}. However, the results of \cite{CianchiFerone} do not apply here, as they are given in a different functional setting (roughly speaking, with $L^p-L^q$ ($1<p<\infty, q=p/(p-1)$) constraints rather than $L^\infty-L^1$). 
The quantitative estimate \eqref{Eq:EstQuant} can take different forms. 
We refer to \cite[Lemma 7.73]{BrascoDePhilippis}, \cite[Proposition 24]{MRBSIAM}, for similar spectral quantitative inequalities in a different context and to \cite[Theorem 1]{Lemou} where such inequalities are used to study the stability of Vlasov-Poisson systems. The proofs of the inequality in the three latter references revolve around the same idea, that of partial Schwarz rearrangements. Finally, we mention \cite[Proposition 2.7]{casas-Wachsmuth} where similar inequalities are obtained in the context of stability of bang-bang optimal controls, with a more direct proof that we will follow.
\end{remark}

\subsection{Examples and comments on the assumptions}\label{sec:comments}

\paragraph{The Dirichlet-Laplacian operator.}
Take $q=1$ and  $-A_0$ to be the Dirichlet-Laplacian operator defined on its domain $D(A_0)=H^1_0(\Omega)\cap H^2(\Omega)$. It follows from the spectral theorem \cite[Section~6.4]{MR2759829} that $\Hun$ is satisfied. By hypoelliptic analyticity of the Laplacian, the eigenfunctions of $A_0$ are analytic in $\Omega$ (and we take them real-valued) and $\Htrois$ is satisfied. 
Regarding $\Hdeux$, note that the first eigenvalue $\lambda_1(\Omega)$ of $A_0$ is simple (hence, $J_1=\{1\}$) according to the Krein-Rutman theorem and the eigenfunction $\phi_1$ is positive in $\Omega$. 
By \eqref{carac:omegastar}, we have $\omega^* = \{ \phi_1^2>\mu^*\}$ and, by analyticity of $\phi_1$, $\omega^*$ satisfies the interior sphere property (see \cite{Henrot-Pierre}), $\partial\omega$ is at a positive distance of $\partial\Omega$ and, according to the Hopf maximum principle, 
$$
\inf_{\partial\omega}\Vert \nabla \phi_1\Vert>0.
$$
Hence $\Hdeux$ is satisfied. 

\paragraph{The Robin-Laplacian operator.}
Take $q=1$, $\beta>0$ and $-A_0$ to be the Robin-Laplacian operator with boundary condition $\partial_ny+\beta y=0$ on $\partial\Omega$. By similar arguments, $\Hun$ and $\Hdeux$ are satisfied. The analyticity property $\Htrois$ follows for instance from \cite[Section~V.4]{MR969367}.

\paragraph{The Dirichlet Stokes operator in the 2D unit disk.}
Consider the Stokes equation 
$$
\partial_t y -\triangle y + \nabla p=0,\quad\mathrm{div}\, y=0
$$
in $\Omega=\{ x\in\R^2\ \vert\ \Vert x\Vert< 1\}$, the Euclidean unit disk of $\R^2$, with Dirichlet boundary conditions.
The Stokes operator $A_0:D(A_0)\rightarrow H$ is defined by $A_0=-\mathcal{P}\triangle$, with $D(A_0)=\{ y\in V \ \vert\ A_{0}y\in H \}$, with $V = \{ y\in (H^1_0(\Omega))^2\ \vert\ \mathrm{div}\, y=0 \}$, 
$H = \{ y\in (L^2(\Omega))^2\ \vert\ \mathrm{div}\, y=0 ,\ y_{\vert\partial\Omega}.n=0  \}$
and $\mathcal{P}:(L^2(\Omega))^2 = H\overset{\perp}{\oplus}H^\perp\rightarrow H$ the Leray projection (see \cite{MR2986590}).
The first eigenfunction is given by 
\begin{equation*}
\phi_{0,1}(r,\theta) = \frac{-J_0'(\sqrt{\lambda_{0,1}} r) }{ \sqrt{\pi} \sqrt{\lambda_{0,1}} \vert J_0(\sqrt{\lambda_{0,1}})\vert }  \begin{pmatrix} -\sin\theta \\ \cos\theta \end{pmatrix} ,
\end{equation*}
in polar coordinates $(r,\theta)$ (see \cite{MR2480708,MR1906228}), where $J_0$ is the Bessel function of the first kind of order $0$ and $\lambda_{0,1}=z_{1,1}^{2}$ with $z_{1,1}>0$ is the first positive zero of $J_{0}$.
Using that $\phi_{0,1}$ is radially-symmetric, it follows from properties of the Bessel functions that $\Hun$, $\Htrois$ and $\Hdeux$ are satisfied. 

\paragraph{System of coupled heat equations in the cube.}
Setting $Y=(y_1,y_2,y_3)^\top$, consider the system of coupled 1D heat equations
$$
\partial_tY-\triangle Y+AY=0, \qquad 
$$
in $\Omega=(0,1)^3$, with Dirichlet boundary conditions, where $\triangle Y = (\triangle y_1,\triangle y_2,\triangle y_3)^\top$, where $A\in M_3(\C)$ is a $3\times 3$ matrix with three distinct complex eigenvalues $\mu_1$, $\mu_2$, $\mu_3$ and such that $-\pi^2<\Re \mu_1=\Re \mu_2<\Re \mu_3$. Let $E_i=\operatorname{Span}(u_i)$, with $u_i\in\C^3$, be the eigenspace of $A$ associated to $\mu_i$, for $i\in\{1,2,3\}$. Set $\rho=\Re \mu_1$.

The above system is known to be observable in any time $T>0$ on any nonempty open subset $\omega$ of $\Omega$. We refer for instance to \cite[Theorem~1.1]{ammar2009kalman} or \cite[Proof of the first point of Theorem~1]{lissy2019internal}.

Here, we have $q=3$, $A_0=\triangle+A\operatorname{Id}$ on $D(A_0)=(H^2(\Omega)\cap H_0^1(\Omega))^3$, $J_1=\{1,2\}$, $p_0=3$ and $\gamma_1(T)=(e^{2\rho T}-1)/(2\rho)$. The assumptions $\Hun$ and $\Htrois$ are satisfied, and
$\phi_1(x)=\sin(\pi x_1)\sin (\pi x_2)u_1$ and $\phi_2(x)=\sin(\pi x_1)\sin (\pi x_2)u_2$.
We define
$$
\Phi(x)=\eta \sin^2(\pi x_1)\sin^2(\pi x_2)\left|\beta_{11}u_1+\beta_{12}u_2\right|^2+(1-\eta) \sin^2(\pi x_1)\sin^2(\pi x_2)\left|\beta_{21}u_1+\beta_{22}u_2\right|^2,
$$
where $\eta \in [0,1]$ and $\beta_{ij}\neq 0$. Reasoning as before, and using that $\beta_{11}u_1+\beta_{12}u_2\neq 0$ and $\beta_{21}u_1+\beta_{22}u_2\neq 0$ (by linear independence of $u_1$ and $u_2$), we get that $|\nabla \Phi|$ is uniformly bounded below on any nontrivial level set of $\Phi$, and thus $\Hdeux$ is satisfied.

\paragraph{On the sharpness of $\Hdeux$: periodic Laplacian in the 1D torus.} 
Take $\Omega=(0,2\pi)$, $q=1$ and $A_0=-d^2/dx^2$ with periodic boundary conditions $y(0)=y(2\pi)$ and $y'(0)=y'(2\pi)$.
The first eigenvalue is $\lambda_1=1$ with eigenspace spanned by $x\mapsto \cos x$ and $x\mapsto \sin x$. For every $a\in \overline{U}_L$, 
we have
$$
M_1(a)=\begin{pmatrix}
\int_0^{2\pi} a(x)\cos^2x\, dx & \int_0^{2\pi} a(x)\cos x\sin x\, dx\\[1mm]
\int_0^{2\pi} a(x)\cos x\sin x\, dx & \int_0^{2\pi} a(x)\sin^2x\, dx 
\end{pmatrix},
$$
and thus
$$
\sigma_1(a)=\frac12 \int_0^{2\pi}a(x)\, dx- \frac12 \sqrt{\left(\int_0^{2\pi}a(x)\cos (2x)\, dx\right)^2+\left(\int_0^{2\pi} a(x)\sin(2x)\, dx\right)^2} .
$$
Therefore, the optimal design problem \eqref{max:sigma1} (maximizing $\sigma_1$ over $\overline{\mathcal{U}}_L$) is equivalent to
$$
\max_{a\in\overline{\mathcal{U}}_L}\left(\int_0^{2\pi}a(x)\cos (2x)\, dx\right)^2+\left(\int_0^{2\pi} a(x)\sin(2x)\, dx\right)^2.
$$
It is easy to see that the optimal value for this problem is 0, meaning that the maximal value of $\sigma_1$ is $\frac{L\pi}{2}$. Moreover, the maximal value is reached for any function $a$ given by
$$
a(x)=\frac{L}{2}+\eta \bigg(a_1^c\cos x+a_1^s\sin x+\sum_{k=3}^m (a_k^c\cos (kx)+a_k^s\sin (kx))\bigg)
$$ 
for any $m\in \overline{\N}\setminus \{0,1,2\}$ and any families of real numbers $(a_k^s)_{1\leq k\leq m}$ and $(a_k^c)_{1\leq k\leq m}$, with $\eta$ chosen small enough such that $a\in\overline{\mathcal{U}}_L$.
Hence, there exist maximizers in $\overline{\mathcal{U}}_L\setminus\mathcal{U}_L$ and the conclusion of Theorem~\ref{theo:sigma1} is not true in that case. The main reason is that $\Hdeux$ is not satisfied here. Indeed, one has $\# J_1=2$ and it is possible to choose adequately the coefficients defining $\Phi$ so that $\Phi(x)=\cos^2x+\sin^2x=1$, and its normal derivative vanishes on $\partial\omega$.

\subsection{Comments on the strategy of proof}\label{sec_comments_proof}
In this section, we describe our strategy of proof, highlighting the difficulties that will be encountered, on the example of the Dirichlet heat equations. We take $q=1$ and $-A_0$ equal to the Dirichlet-Laplacian on $D(A_0)=H^1_0(\Omega)\cap H^2(\Omega)$. The operator $A_0$ is selfadjoint, of compact resolvent and then its spectrum is real and the first eigenvalue $\lambda_1(\Omega)$ is simple, and we take real-valued eigenfunctions. 

Recalling that $\bar\delta_T$ is the optimal value for Problem~\eqref{maxCTa}, defined by \eqref{maxCTa}, considering particular choices of coefficients $b_j$ in \eqref{defCTdeterministic}, we have
\begin{eqnarray}\label{train0932}
\bar\delta_T&\leq & \gamma_1(T) \max_{\mathds{1}_\omega \in \mathcal{U}_L}\int_\Omega a(x)\phi_1(x)^2\, dx \qquad\forall a\in\overline{\mathcal{U}}_L
\end{eqnarray}
(see the beginning of Section~\ref{prooftheo1} for details).

Let us now establish a lower bound for $\bar\delta_T$.
Let $a_{1}^T\in\overline{\mathcal{U}}_L$ be the unique solution to the shape optimization problem on the right-hand side of \eqref{train0932}.
We define
$$
E_{T,a}(b)=\int_0^T \! \int_\Omega a(x)\bigg(\sum_{j=1}^{+\infty}b_je^{\lambda_j t}\phi_j(x)  \bigg)^2 dx \, dt
$$
and we set $\varepsilon=|b_1|^2$ and $c_j=b_j/\sqrt{1-\varepsilon}$ for $j\geq 2$, so that $\sum_{j=2}^{+\infty}\vert c_j\vert^2=1$. A straightforward computation gives 
$$
E_{T,a}(b)=\varepsilon A_{T,a}(c)+(1-\varepsilon) B_{T,a}(c)+2\sqrt{\varepsilon(1-\varepsilon)}D_{T,a}(c)
$$
where
\begin{eqnarray*}
A_{T,a}(c) &=&  \gamma_1(T) \int_\Omega a\phi_1^2 ,  \\
B_{T,a}(c) &=&\int_0^T\!  \int_\Omega a(x)\bigg(\sum_{j=2}^{+\infty} c_je^{\lambda_j t}\phi_j(x)  \bigg)^2 \, dx dt , \\
D_{T,a}(c) &=& \int_0^T \int_\Omega a(x)e^{\lambda_1t}\phi_1(x)\sum_{j=2}^{+\infty}c_je^{\lambda_jt}\phi_j(x) \, dx\, dt = \sum_{j=2}^{+\infty}c_j\frac{e^{(\lambda_j+\lambda_1) T}-1}{\lambda_1+\lambda_j} \int_\Omega a\phi_1\phi_j .
\end{eqnarray*}
From the maximality of $\bar\delta_T$, we have
$$
\bar\delta_T\geq \inf_{\sum_{j=2}^{+\infty}c_j^2=1}\min_{\varepsilon\in [0,1]} \varepsilon A_{T,a}(c)+(1-\varepsilon) B_{T,a}(c)+2\sqrt{\varepsilon(1-\varepsilon)}D_{T,a}(c)
$$ 
Due to the presence of exponential terms, we could expect to be able to exploit this decomposition to prove that the first mode dominates the others when $T$ is sufficiently large. 
This does not work however, essentially because the modes $j\geq2$ in the decomposition can modify the observability constant in large time. To understand this difficulty, note that, by the Cauchy-Schwarz inequality, 
$$
|D_{T,a_1^T}(c)| \leq  \sqrt{A_{T,a_1^T}(c)B_{T,a_1^T}(c)} .
$$
Therefore
$$
\bar\delta_T \geq \min_{\varepsilon\in[0,1]} \left(  \varepsilon A_{T,a_1^T}(c)+(1-\varepsilon) B_{T,a_1^T}(c)-2\sqrt{\varepsilon(1-\varepsilon)}\sqrt{A_{T,a_1^T}(c)B_{T,a_1^T}(c)} \right) = 0.
$$
The latter equality follows from the fact that the minimum is equal to the first eigenvalue of the matrix
$$
\begin{pmatrix}A_{T,a_1^T}(c) & \sqrt{A_{T,a_1^T}(c)B_{T,a_1^T}(c)} \\ \sqrt{A_{T,a_1^T}(c)B_{T,a_1^T}(c)} & B_{T,a_1^T}(c)\end{pmatrix}.
$$
Unfortunately, this estimate is not useful. The difficulty here is due to the presence of crossed terms that are difficult to handle. To overcome it, we will consider a particular path of ``quasi-maximizers'', replacing $a_1^T$ by a convex combination of the expected maximizer in large time and the constant function equal to $L$. We will choose appropriate convexity weights, to give an increasing importance to the term $a_1^T$ as $T\rightarrow+\infty$.

\section{Proofs}\label{sec_proof}
\subsection{Proof of Theorem~\ref{theo:sigma1}}\label{prooftheo2}
First of all, noting that, for every $a\in\overline{\mathcal{U}}_L$,
$$
\sigma_1(a)=\min_{\sum_{j\in J_1}|b_j|^2=1}\int_\Omega a(x)\left|\sum_{j\in J_1}b_j\phi_j(x)\right|^2\, dx 
$$
is as minimum of linear continuous functionals, $\sigma_1$ is upper semicontinuous for the weak star topology of $L^\infty$. By compactness of $\overline{\mathcal{U}}_L$ for this topology, it follows that \eqref{max:sigma1} has at least one solution $a_1$.

We are going to prove that $a_1$ is unique and that $a_1=\mathds{1}_{\omega^*}$ for some subset $\omega^*$ (``bang-bang property").

\paragraph{Bang-bang property.} In order to prove that $a_1$ is actually the characteristic function of some measurable subset, we exploit first-order optimality conditions.
From \cite[Theorem~1]{MR1351641}, since $\sigma_1(a)$ has a finite multiplicity, the mapping $\overline{\mathcal{U}}_L\ni a \mapsto \sigma_1(a)$ is sub-differentiable at $a_1$. Let $(b^k)_{k\in \widetilde{J}_1}$ be an orthonormal basis of the eigenspace associated to $\sigma_1(a_1)$, with $\widetilde{J}_1\subset J_1$. 
We have
$$
\partial \sigma_1(a_1)=\operatorname{co}\left\{\partial j_k(a_1)\ \mid \ k\in \widetilde{J}_1\right\}\quad \text{with}\quad j_k(a)=\int_\Omega a(x)\left|\sum_{j\in \widetilde{J}_1}b_j^k\phi_j(x)\right|^2\, dx,
$$ 
where ``co'' means ``convex hull''. 
We consider the tangent cone\footnote{The tangent cone $\mathcal{T}_{a_1}$ is the set of functions $h\in L^\infty(\Omega)$ such that, for any sequence of positive real numbers $\varepsilon_n$ decreasing to $0$, there exists a sequence of functions $h_n\in L^\infty(\Omega)$ converging to $h$ as $n\rightarrow +\infty$, and $a_1+\varepsilon_nh_n\in \overline{\mathcal{U}}_l$ for every $n\in\N$ (see \cite{MR1367820} for example).} $\mathcal{T}_{a_1}$ to the set $\overline{\mathcal{U}}_L$ at $a_1$, and the indicator function $\iota_{[0,1]}$ given by
$$
\iota_{\overline{\mathcal{U}}_L}(u)=\left\{\begin{array}{ll}
0 & \text{if } u\in\overline{ \mathcal U}_L,\\
+\infty & \text{otherwise}.
\end{array}\right.
$$
Finally, we define the linear functional $j_0:\overline{\mathcal{U}}_L\ni a \mapsto \int_\Omega a$.
The optimization problem \eqref{max:sigma1} can be recast as
$$
\min_{\substack{a\in L^\infty(\Omega)\\ j_0(a)=L|\Omega|}} \left( -\sigma_1(a)+\iota_{\overline{\mathcal{U}}_L}(a) \right).
$$
If $a_1$ is an optimal solution, according to the Lagrange multiplier rule, there must exist $\mu^*\in \R$ such that
$$
0\in \partial (-\sigma_1)(a_1)+\mu^*\partial j_0(a_1)+\partial \iota_{\overline{\mathcal{U}}_L}(a_1).
$$
Equivalently, there exists $(\alpha_k)_{k\in \widetilde{J}_1}\in \R_+^{\# \widetilde{J}_1}$ such that $\sum_{k\in \widetilde{J}_1}\alpha_k=1$ and for every $h\in L^\infty(\Omega)$ such that $0\leq a_1+\eta h \leq 1$ for $\eta>0$ small enough,
\begin{equation}\label{strasb:1750}
-\sum_{k\in \widetilde{J}_1}\alpha_k \langle dj_k(a_1),h\rangle+\langle dj_0(a_1),h\rangle=0\Longleftrightarrow \int_\Omega h(x)(\Psi(x)-\mu^*)\, dx\geq 0,
\end{equation}
where the function $\Psi:\Omega\to \R_+$ is defined by
\begin{equation*}
\Psi(x)=\sum_{k\in \widetilde{J}_1}\alpha_k\bigg|\sum_{j\in \widetilde{J}_1}b_j^k\phi_j(x)\bigg|^2.
\end{equation*}
Let us now prove that $a_1$  is necessarily {\it bang-bang}, i.e., that $a_1=\mathds 1_{\omega^*}$ for some Lebesgue measurable $\omega^*\subset \O$. By contradiction, assume that $|\{0<a_1<1\}|>0$. 
Let $x_0$ be a Lebesgue point of  $\{0<a_1<1\}$. There exists a sequence $(G_{n})_{n\in \N}$ of measurable subsets with $G_{n}\subset \{0<a_1<1\}$ such that 
\begin{equation}\label{x0:density1}
\lim_{\varepsilon \to 0}\frac{|G_n\cap B(x_0,\varepsilon)|}{|B(x_0,\varepsilon|}=1,
\end{equation}
where $B(x_0,\varepsilon)$ denotes the ball centered at $x_0$ with radius $\varepsilon$ in $\R^d$.
Setting $h=\mathds{1}_{G_{n}}$ and noting that $0\leq a_1 \pm\eta h\leq 1$  if $\eta$ is small enough, we infer from \eqref{strasb:1750} that
\[
\pm \int_{G_n} (\Psi(x)-\mu^*)\, dx\geq 0.
\] 
Dividing this inequality by $|G_{n}|$ and letting $G_{n}$ shrink to $\{x_0\}$ as $n\to +\infty$ we infer that $\Psi(x)=\mu^*$ a.e. in $\{0<a_1<1\}$.
Since $0<|\{0<a_1<1\}|\leq |\{\Psi=\mu^*\}|$ and $\Psi$ is analytic, we must have $\Psi=\mu^*$ and $\nabla \Psi=0$ everywhere in $\Omega$, which contradicts $\Hdeux$.

We thus infer that $|\{0<a_1<1\}|=0$ and hence $a_1=\mathds{1}_{\omega^*}$ for some Lebesgue measurable subset $\omega^*$ of $\Omega$ such that $|\omega^*|=L|\Omega|$.

\paragraph{Characterization of $\omega^*$.}
Consider a Lebesgue point $x_0$  of  $\omega^*$. Let $(G_{n})_{n\in \N}$ be a sequence of measurable subsets such that $G_{n}\subset\omega^*$ and \eqref{x0:density1} holds. Setting $h=-\mathds{1}_{G_{n}}$ and noting that $0\leq a_1 +\eta h\leq 1$ if $\eta$ is small enough, we infer from \eqref{strasb:1750} that
\[
\int_{G_n} (\Psi(x)-\mu^*)\, dx\geq 0.
\] 
Dividing this inequality by $|G_{n}|$ and letting $G_{n}$ shrink to $\{x_0\}$ as $n\to +\infty$ yields $\Psi\geq \mu^*$ a.e. in $\omega^*$. Similarly, $\Psi\leq \mu^*$ in $(\omega^*)^c$. We conclude that $\omega^*=\{\Psi >\mu^*\}$ by noting that $\{\Psi=\mu^*\}$ has zero Lebesgue measure.

\paragraph{Uniqueness and regularity of $\omega^*$.}

Assume by contradiction that Problem~\eqref{max:sigma1} has two distinct maximizers $a_1$ and $a_2$. As a consequence of our previous analysis, there exist $\omega_1\,, \omega_2\subset \O$  such that $a_i=\mathds{1}_{\omega_i}$ ($i=1,2$). By concavity of $\sigma_1$, $(a_1+a_2)/2$ is also a solution of Problem~\eqref{max:sigma1}, which contradicts the bang-bang property of maximizers.

Finally, the regularity property of $\omega^*$ 
follows from $\Htrois$, which implies the analyticity of $\Psi$. Therefore, $\omega^*$ is an open semi-analytic set.

\paragraph{Quantitative estimate.} 
At last, we prove the quantitative estimate \eqref{Eq:EstQuant}; the argument we follow to establish it is inspired and adapted from \cite{casas-Wachsmuth}. Observe that with the same function $\Psi$ we have, for any $a\in \overline{\mathcal U}_L$,
\[ \sigma_1(a)-\sigma_1(a_1)\leq \int_\O (a-a_1)\Psi.\] 
Let us prove that there exists $C>0$ such that
$$
\forall a \in  \overline{\mathcal U}_L, \quad \int_\O (a-a_1)\Psi\leq -C \Vert a-a_1\Vert_{L^1(\O)}^2,
$$
which can be seen as a quantitative version of the bathtub principle. 

We have proved that $a_1=\mathds{1}_{\omega^*}$ for some $\omega^*\subset \O$ such that $|\omega^*|=L|\O|$, and moreover that
$\Psi \geq \mu^*$ in $\omega^*$ and $\Psi \leq \mu^*$ in $\O\backslash \omega^*$.
Noting that for every $a\in \overline{\mathcal U}_L$, $ \int_\O (a-a_1)\Psi= \int_\O (a-a_1)(\Psi-\mu^*)$, that $a-a_1\leq 0$ on $\omega^*$ and $a-a_1\geq 0$ on $\omega^*$, we get that
$$
\int_\O (a-a_1)\Psi=-\int_\O |a-a_1|.|\Psi-\mu^*|\leq -\int_{\O_\delta} |a-a_1|.|\Psi-\mu^*|
$$ 
where, for a given $\delta >0$, the set $\O_\delta=\{|\Psi-\mu^*|\geq \delta\}$ is the complement of a tubular neighborhood of the levet set $\{\Psi=\mu^*\}$.
By definition of $\O_\delta$, it follows that 
\begin{multline}\label{eq1701}
\int_\O (a-a_1)\Psi \leq -\delta \Vert a-a_1\Vert_{L^1(\Omega_\delta)}
= -\delta (\Vert a-a_1\Vert_{L^1(\Omega)}-\Vert a-a_1\Vert_{L^1(\Omega\backslash \Omega_\delta)}) \\
\leq  -\delta (\Vert a-a_1\Vert_{L^1(\Omega)}-|\Omega\backslash \Omega_\delta|)
\end{multline}

\begin{lemma}\label{lem:1704}
There exists $M>0$ such that $|\Omega\backslash \Omega_\delta|\leq M\delta $ for every $\delta>0$. 
\end{lemma}

\begin{proof}
It suffices to prove this estimate for $\delta>0$ small enough. We establish the existence of $K>0$ such that if $s>0$ is small enough then $ \{\Psi=\mu^*+s\}\subset \Sigma_{\mu^*}+K\delta \mathbb{B}$
where $\Sigma_{\mu^*}=\{\Psi=\mu^*\}$ and $\mathbb{B}$ is the centered unit ball. 
Assuming such a $K$ exists, we infer that$$
\Omega\backslash \Omega_\delta =\bigcup_{s=-\delta}^\delta \{\Psi=\mu^*+s\}\subset \Sigma_{\mu^*}+K\delta \mathbb{B}.
$$
Now, using the definition of the perimeter of rectifiable curves via the Minkowski content, we obtain 
$|\Omega\backslash \Omega_\delta |\leq 2\delta \operatorname{Per}(\Sigma_{\mu^*})$ for $\delta$ small enough,
whence the conclusion with $M=2 K \operatorname{Per}(\Sigma_{\mu^*})$. It thus remains to show that 
$$
\operatorname{dist}_H(\Sigma_{\mu^*},\{\Psi=\mu^*+s\})\leq Ms
$$
for $s>0$ small enough, where $\operatorname{dist}_H$ is the Hausdorff distance between two compact sets.

By contradiction, assume that there exist two sequences $(x_k)_{k\in\N}$ in $\Sigma_{\mu^*}$ and $(y_k)_{k\in\N}$ in $\Sigma_{\mu^*+s_k}$, where $(s_k)_{k\in \N}$ denotes a sequence converging to 0, such that 
$$
|x_k-y_k|=\operatorname{dist}_H(\Sigma_{\mu^*+s_k},\Sigma_{\mu^*}) \quad \text{and}\quad 
\lim_{k\to +\infty}\frac{|\Psi (x_k)-\Psi(y_k)|}{|x_k-y_k|}=0.
$$
By the mean value theorem, there exists $z_k$ belonging to the segment $[x_k,y_k]$ such that 
$$
\frac{|\Psi (x_k)-\Psi(y_k)|}{|x_k-y_k|}=\left\langle \nabla \Psi(z_k),\frac{y_k-x_k}{|y_k-x_k|}\right\rangle.
$$
Since $\Sigma_{\mu^*}$ is $\mathscr{C}^2$, it satisfies the uniform interior ball property, and therefore $y_k=x_k+t_k\nu(x_k)+\operatorname{o}(t_k)$ for $k$ large enough, with $t_k\to 0$, where $\nu(x_k)$ denotes the outward unit normal vector on $\Sigma_{\mu^*}=\partial \{\Psi \leq \mu^*\}$. We infer that 
$\left\langle \nabla \Psi(z_k),\nu(x_k)+\operatorname{o}(1)\right\rangle\to 0$ as $k\to +\infty$.
This contradicts $\Hdeux$. 
\end{proof}

Given any $a\in \overline{\mathcal U}_L$, we set $\delta =\frac{1}{2M}\Vert a-a_1\Vert_{L^1(\O)}$.
By \eqref{eq1701} and by Lemma~\ref{lem:1704}, we have
$$
\int_\O (a-a_1)\Psi \leq -\frac{1}{2M}\Vert a-a_1\Vert_{L^1(\O)}\left(\Vert a-a_1\Vert_{L^1(\O)}-\frac{M}{2M}\Vert a-a_1\Vert_{L^1(\O)}\right)
=  -\frac{1}{4M}\Vert a-a_1\Vert_{L^1(\O)}^2.
$$
The conclusion follows.

\subsection{Proof of Theorem~\ref{theo:largeT}}\label{prooftheo1}
Considering particular choices of coefficients $b_j$ equal to 0 in $\N^*\backslash J_1$ in \eqref{defCTdeterministic}, we have
\begin{eqnarray}
\bar\delta_T&\leq &\gamma_1(T) \sup_{a\in \overline{\mathcal{U}}_L}\min_{\sum_{j\in J_1}|b_j|^2=1}\int_\Omega a(x)\bigg|\sum_{j\in J_1}b_j\phi_j(x)\bigg|^2\, dx\nonumber \\
&=& \gamma_1(T) \max_{\mathds{1}_\omega \in \mathcal{U}_L}\min_{\sum_{j\in J_1}|b_j|^2=1}\int_\Omega a(x)\bigg|\sum_{j\in J_1}b_j\phi_j(x)\bigg|^2\, dx\nonumber \\
&=& \gamma_1(T) \max_{\mathds{1}_\omega \in \mathcal{U}_L}\sigma_1(\mathds{1}_\omega).
\label{eq1212}
\end{eqnarray}
To go from the first to the second line, we used Theorem \ref{theo:sigma1}. The rest of the proof is devoted to establishing a lower bound on $\bar\delta_T$.
We define
$$
E_{T,a}(b)=\int_0^T \! \int_\Omega a(x)\bigg|\sum_{j=1}^{+\infty}b_je^{\lambda_j t}\phi_j(x)  \bigg|^2 dx \, dt.
$$
and we set $\varepsilon=\sum_{j\in J_1}|b_j|^2$  (so that $1-\varepsilon=\sum_{j\in \N^*\backslash J_1}|b_j|^2$). We also define sequence $(c_j)_{\in \N^*}$ in $\ell^2(\R)$ by 
$$
c_j=\left\{\begin{array}{ll}
b_j/\sqrt{\varepsilon} & \text{if }j\in J_1\\
b_j/\sqrt{1-\varepsilon} & \text{otherwise}
\end{array}\right. 
$$
so that
$$
\sum_{j\in J_1}|c_j|^2=\sum_{j\in \N^*\setminus J_1}|c_j|^2=1.
$$
We define the two subsets of $\ell^2(\C)$
$$
\Lambda_{J_1}=\bigg\{(c_j)\in \C^{\# J_1}, \ \sum_{j\in J_1}|c_j|^2=1\bigg\}\quad \text{and}\quad \Lambda_{\N^*\setminus J_1}=\bigg\{(c_j)\in \ell^2(\N^*\setminus J_1), \ \sum_{j\in \N^*\setminus J_1}|c_j|^2=1\bigg\}.
$$
A straightforward computation shows that
$$
E_{T,a}(b)=\varepsilon A_{T,a}(c)+(1-\varepsilon) B_{T,a}(c)+2\sqrt{\varepsilon(1-\varepsilon)}D_{T,a}(c),
$$
where
\begin{eqnarray*}
A_{T,a}(c) &=& \gamma_1(T) \int_\Omega a(x)\bigg|\sum_{j\in J_1}c_j\phi_j(x)\bigg|^2\, dx\\
B_{T,a}(c) &=&\int_0^T\!  \int_\Omega a(x)\bigg|\sum_{j\in \N^*\backslash J_1} c_je^{\lambda_j t}\phi_j(x)  \bigg|^2 \, dx dt\\
D_{T,a}(c) &=& \Re \int_\Omega\!\int_0^T a(x)\sum_{k\in J_1}c_ke^{\lambda_1t}\phi_k(x)\overline{\sum_{j\in \N^*\backslash J_1}c_je^{\lambda_jt}\phi_j(x)} \, dtdx
\end{eqnarray*}
Consequently,
\begin{eqnarray*}
C_T(a) & = & \inf_{\sum_{j\in\N^*}\vert b_j\vert^2=1}E_{T,a}(b)\\
&=& \inf_{\varepsilon\in [0,1]}\inf_{(c_j)\in \Lambda_{J_1}}\bigg(\varepsilon A_{T,a}(c)+\inf_{(c_j)\in \Lambda_{\N^*\setminus J_1}}(1-\varepsilon) B_{T,a}(c)+2\sqrt{\varepsilon(1-\varepsilon)}D_{T,a}(c)\bigg)
\end{eqnarray*}
Moreover, note that
$$
\inf_{(c_j)\in \Lambda_{J_1}}A_{T,a}(c)=\gamma_1(T)\sigma_1(a)
$$
Let us consider the unique solution $a_1=\mathds{1}_{\omega^*}$ of \eqref{max:sigma1}.
From now on, with a slight abuse of notation, we will use the notation $A_{T,a_1}=\gamma_1(T)\sigma_1(a_1)$.

Let $\nu\in (0,1)$ to be chosen later, and define $a_\nu=\nu a_1+(1-\nu)L\in\overline{\mathcal{U}}_L$ (by convexity). By maximality of $\bar\delta_T$, we have
\begin{eqnarray}
\bar\delta_T & \geq & \inf_{\varepsilon\in [0,1]}\inf_{(c_j)\in \Lambda_{\N^*\setminus J_1}}(\varepsilon A_{T,a_\nu}(c)+(1-\varepsilon) B_{T,a_\nu}(c)+2\sqrt{\varepsilon(1-\varepsilon)}D_{T,a_\nu}(c))\nonumber \\
&\geq & \inf_{\varepsilon\in [0,1]}\inf_{(c_j)\in \Lambda_{\N^*\setminus J_1}}(\nu \varepsilon A_{T,a_1}+\nu(1-\varepsilon) B_{T,a_1}(c)+(1-\nu)(1-\varepsilon) B_{T,L}(c)+2\nu\sqrt{\varepsilon(1-\varepsilon)}D_{T,a_1}(c)).\nonumber 
\label{m1122}
\end{eqnarray}
To obtain this estimate, we have used that $D_{T,a_\nu}(c)=\nu D_{T,a_1}(c)$, since
\[
D_{T,L}(c)=L\Re \sum_{\substack{j\in \N^*\backslash J_1\\ k\in J_1}}\overline{c_j}c_k\frac{e^{(\overline{\lambda_j}+\lambda_1) T}-1}{\lambda_1+\overline{\lambda_j}}\int_\Omega \phi_k(x)\cdot\overline{\phi_j(x)} \, dx =0
\]
by orthogonality of the functions $\phi_j$ in $L^2(\Omega,\C^q)$.
%

We have
\begin{eqnarray}
\lefteqn{\inf_{(c_j)\in \Lambda_{\N^*\setminus J_1}}(\nu \varepsilon A_{T,a_1}+\nu(1-\varepsilon) B_{T,a_1}(c)+(1-\nu)(1-\varepsilon) B_{T,L}(c)+2\nu\sqrt{\varepsilon(1-\varepsilon)}D_{T,a_1}(c))} \nonumber \\
& =& \nu \varepsilon A_{T,a_1}+\inf_{(c_j)\in \Lambda_{\N^*\setminus J_1}} \nu(1-\varepsilon) B_{T,a_1}(c)+(1-\nu)(1-\varepsilon) B_{T,L}(c)+2\nu\sqrt{\varepsilon(1-\varepsilon)}D_{T,a_1}(c)\nonumber \\
 &\geq & \displaystyle  \nu \varepsilon A_{T,a_1}+ \inf_{(c_j)\in \Lambda_{\N^*\setminus J_1}} \Big( (1-\nu)(1-\varepsilon) B_{T,L}(c)+2\nu\sqrt{\varepsilon(1-\varepsilon)}D_{T,a_1}(c)\Big) \nonumber \\
 &\geq &  \nu \varepsilon A_{T,a_1}+(1-\nu)(1-\varepsilon)\frac{L\gamma_{p_0}(T)}{2}\nonumber \\
 && +\inf_{(c_j)\in \Lambda_{\N^*\setminus J_1}} \left( \frac{(1-\nu)(1-\varepsilon)}{2} B_{T,L}(c)+2\nu\sqrt{\varepsilon(1-\varepsilon)}D_{T,a_1}(c) \right)
 \label{m1108}
\end{eqnarray}
by using that 
$$
B_{T,L}(c)=L \sum_{j\in \N^*\backslash J_1}|c_j|^2\int_0^T e^{2\Re\lambda _jt}\, dt=L  \sum_{j=p_0}^{+\infty}\gamma_j(T)|c_j|^2
$$
and that 
$\sum_{j=p_0}^{+\infty}\gamma_j(T)|c_j|^2\geq \gamma_{p_0}(T)\sum_{j=p_0}^{+\infty}|c_j|^2=\gamma_{p_0}(T)$.

\begin{lemma}\label{lem:m1202}
Let $T>0$ and $\mu\in (0,1)$ be given. We define 
\begin{equation}\label{m2042}
\varepsilon_{\nu,T,L}=\frac{L^2(1-\nu)^2\gamma_{p_0}(T)}{16\nu^2\gamma_1(T)+L^2(1-\nu)^2\gamma_{p_0}(T)} \in (0,1) .
\end{equation}
For every $\varepsilon \in [0,\varepsilon_{\nu,T,L}]$, we have
$$
\inf_{(c_j)\in \Lambda_{\N^*\setminus J_1}} \Big( \frac{(1-\nu)(1-\varepsilon)}{2} B_{T,L}(c)+2\nu\sqrt{\varepsilon(1-\varepsilon)}D_{T,a_1}(c) \Big) \geq 0.
$$
%
\end{lemma}
\begin{proof}[Proof of Lemma~\ref{lem:m1202}.]
 By the Cauchy-Schwarz inequality, we have
\begin{eqnarray*}
|D_{T,a_1}(c)|^2 & =& \bigg|\Re \int_\Omega\!\int_0^T a_1(x)\sum_{k=1}^{p_0-1}c_ke^{\lambda_1t}\phi_k(x)\overline{\sum_{j=p_0}^{+\infty}c_je^{\lambda_jt}\phi_j(x)}\, dt\,dx\bigg|^2\\
&\leq & \int_\Omega\!\int_0^T a_1(x)^2\bigg|\sum_{k=1}^{p_0-1}c_ke^{\lambda_1t}\phi_k(x)\bigg|^2\, dt\,dx\int_\Omega\!\int_0^T \bigg|\sum_{j=p_0}^{+\infty}c_je^{\lambda_jt}\phi_j(x)\bigg|^2\, dt\,dx\\
&\leq & \sum_{k=1}^{p_0-1}\gamma_1(T)|c_k|^2\sum_{j=p_0}^{+\infty}\gamma_j(T)|c_j|^2=\gamma_1(T)\sum_{j=p_0}^{+\infty}\gamma_j(T)|c_j|^2.
\end{eqnarray*}
To go from the second to the third line we used the fact that $0\leq a_1(\cdot)\leq 1$.
Therefore,
$$
\frac{(1-\nu)(1-\varepsilon)}{2} B_{T,L}(c)+2\nu\sqrt{\varepsilon(1-\varepsilon)}D_{T,a_1}(c) 
\geq \frac{(1-\nu)(1-\varepsilon)}{2} LX_T(c)^2-2\nu\sqrt{\varepsilon(1-\varepsilon)}\sqrt{\gamma_1(T)}X_T(c)
$$
where 
$$
X_T(c)=\bigg(\sum_{j=p_0}^{+\infty}\gamma_j(T)|c_j|^2\bigg)^{1/2}.
$$
Since $X_T(c)\geq \sqrt{\gamma_{p_0}(T)}$, the expected conclusion follows if 
$$
\sqrt{\gamma_{p_0}(T)}\geq \frac{4\nu \sqrt{\varepsilon}\sqrt{\gamma_1(T)}}{L(1-\nu)\sqrt{1-\varepsilon}},
$$
which is satisfied for any $\varepsilon\in [0,\varepsilon_{\nu,T,L}]$ where $\varepsilon_{\nu,T,L}$ is defined by \eqref{m2042}.
\end{proof}

\begin{lemma}\label{Le:LE}
One has
\begin{equation}\label{m1217}
\bar\delta_T\geq \inf_{\varepsilon\in [0,1]}\psi_{T,\nu,L}(\varepsilon),
\end{equation}
where 
\begin{equation}\label{m1217bis}
\psi_{T,\nu,L}(\varepsilon)=\left\{\begin{array}{ll}
\min \left(\nu A_{T,a_1},(1-\nu)\frac{L\gamma_{p_0}(T)}{2}\right) & \text{if }\varepsilon \in [0,\varepsilon_{\nu,T,L}]\\
(1-\nu)(1-\varepsilon)\gamma_{p_0}(T) & \text{if }\varepsilon \in (\varepsilon_{\nu,T,L},1].
\end{array}\right.
\end{equation}
\end{lemma}

\begin{proof}[Proof of Lemma \ref{Le:LE}]
We combine Lemma~\ref{lem:m1202} and \eqref{m1108}: first, observe that for any $\varepsilon\in [0,1]$ 
$$
\nu \varepsilon A_{T,a_1}+(1-\nu)(1-\varepsilon)\frac{L\gamma_{p_0}(T)}{2}\geq \min \left(\nu A_{T,a_1},(1-\nu)\frac{L\gamma_{p_0}(T)}{2} \right) .
$$
Second,  the expression of $\psi_{T,\nu,L}$ on $(\varepsilon_{\nu,T,L},1]$ follows from \eqref{m1122}, using that $E_{T,a_1}(b)\geq 0$ for all $b\in \ell^2(\R)$ and the fact that
$$
\nu \varepsilon A_{T,a_1}+\nu(1-\varepsilon) B_{T,a_1}(c)+(1-\nu)(1-\varepsilon) B_{T,L}(c)+2\nu\sqrt{\varepsilon(1-\varepsilon)}D_{T,a_1}(c)\geq (1-\nu)(1-\varepsilon) B_{T,L}(c)
$$
for all $c\in \ell^2(\R)$.
\end{proof}

The end of the proof consists in choosing $\nu$ and $\varepsilon$ adequately to compute $\inf_{\varepsilon\in [0;1]} \psi_{T,\nu,L}(\varepsilon).$
\begin{lemma}\label{lem:nuTdef}
There exists a family $(\nu_T)_{T>0}$ converging to 1 as $T\to +\infty$ such that $\nu_T\in (0,1)$ for every $T>0$, and
$$
\lim_{T\to +\infty}\varepsilon_{\nu_T,T,L}=1,\qquad \lim_{T\to +\infty }\frac{(1-\nu_T)\gamma_{p_0}(T)}{\gamma_1(T)}=+\infty, \qquad \lim_{T\to +\infty }\frac{(1-\nu_T)(1-\varepsilon_{\nu_T,T,L})\gamma_{p_0}(T)}{\gamma_1(T)}=+\infty. 
$$
\end{lemma}

\begin{proof}
Let $(\nu_T)_{T>0}$ be a family of elements in $[0,1]$ converging to 1 as $T\to +\infty$. Let $\theta_T=1-\nu_T$. Since $\gamma_{p_0}(T)/\gamma_1(T)\sim e^{2T\Re(\lambda_{p_0}-\lambda_1)}$ as $T\to +\infty$, let us set $\theta_T =\gamma_1(T))e^{\eta T}/\gamma_{p_0}(T)$, where $\eta\in \R_+^*$ will be chosen appropriately. Note first that
$$
1-\nu_T=\theta_T\sim e^{(\eta+2\Re(\lambda_1-\lambda_{p_0}))T}\quad\text{as }T\to +\infty\quad  \text{and}\quad \frac{(1-\nu_T)\gamma_{p_0}(T)}{\gamma_1(T)}=e^{\eta T} .
$$
In particular, this indicates that we should choose $\eta \in(0,2\Re(\lambda_{p_0}-\lambda_1))$.

Regarding the last equality of the lemma, note that
$$
1-\varepsilon_{\nu_T,T,L}=\frac{16\nu_T^2\gamma_1(T)}{16\nu^2\gamma_1(T)+L^2(1-\nu_T)^2\gamma_{p_0}(T)}
$$
and thus
$$
(1-\nu_T)(1-\varepsilon_{\nu_T,T,L})\frac{\gamma_{p_0}(T)}{\gamma_1(T)} = \frac{16\nu_T^2(1-\nu_T)\gamma_{p_0}(T)}{16\nu_T^2\gamma_1(T)+L^2(1-\nu_T)^2\gamma_{p_0}(T)}
= \frac{e^{\eta T}}{1+\frac{L^2e^{2\eta T}}{16\nu_T^2}\frac{\gamma_1(T)}{\gamma_{p_0}(T)}} .
$$
Since
$$
e^{2\eta T}\frac{\gamma_1(T)}{\gamma_{p_0}(T)}=\operatorname{O}\left(e^{2T(\eta +\lambda_1-\lambda_{p_0})}\right),
$$
choosing any $\eta$ such that $\eta \in \left(\Re(\lambda_{p_0}-\lambda_1),2\Re(\lambda_{p_0}-\lambda_1)\right)$ yields the desired conclusion. Observe moreover that, with this choice, we have
$$
\varepsilon_{\nu_T,T,L}=\frac{L^2(1-\nu_T)^2\gamma_{p_0}(T)}{L^2(1-\nu_T)^2\gamma_{p_0}(T)+16\nu_T^2\gamma_1(T)}=\frac{L^2e^{2\eta T}\gamma_1(T)/\gamma_{p_0}(T)}{16\nu_T^2+L^2e^{2\eta T}\gamma_1(T)/\gamma_{p_0}(T)} 
= 1+\operatorname{O}\big(e^{-2T(\eta-\Re(\lambda_{p_0}-\lambda_1))}\big)  
$$
since $\eta+\Re(\lambda_1-\lambda_{p_0})>0$.
\end{proof}
\paragraph{Conclusion: asymptotics of $\bar\delta_T$.}
Let us consider $\nu=\nu_T$ in \eqref{m1217}, where $(\nu_T)_{T>0}$ denotes the family constructed in Lemma~\ref{lem:nuTdef}. One has
$$
\bar\delta_T\geq \inf_{\varepsilon\in [0,1]}\psi_{T,\nu_T,L}(\varepsilon)
$$
where $\psi_{T,\nu_T,L}$ is given by \eqref{m1217bis}. From Lemma~\ref{lem:nuTdef}, there holds
$$
\lim_{T\to +\infty}\frac{(1-\nu_T)\gamma_{p_0}(T)L}{2\gamma_1(T)}=+\infty \quad \textrm{and}\quad \lim_{T\to +\infty }\frac{(1-\nu_T)(1-\varepsilon_{\nu_T,T,L})\gamma_{p_0}(T)}{\gamma_1(T)}=+\infty.
$$
From \eqref{m1217} we infer that
\begin{equation}\label{mtgv:0818}
\frac{\bar\delta_T}{\gamma_1(T)}\geq \min \left(\nu_T \sigma_1(a_1),\frac{(1-\nu_T)\gamma_{p_0}(T)L}{2\gamma_1(T)},\frac{(1-\nu_T)(1-\varepsilon_{\nu_T,T,L})\gamma_{p_0}(T)}{\gamma_1(T)}\right).
\end{equation}
Letting $T$ go to $+\infty$ yields 
$$
\liminf_{T\to +\infty}\frac{\bar\delta_T}{\gamma_1(T)}\geq \sigma_1(a_1).
$$
Combining this inequality with \eqref{eq1212} finally gives \eqref{Eq:HP5}.

\paragraph{Quantitative estimate on $\boldsymbol{\operatorname{dist}_{L^1(\Omega)}\Big(a_T^\star,\underset{a\in \overline{\mathcal{U}}_L}{\operatorname{argmax}}~\sigma_1(a)\Big)}$.}
Since $\bar\delta_T=C_T(a_T^\star)$ and since $C_T(a_T^\star)$ is defined as an infimum, considering particular choices of coefficients $b_j$ equal to $0$ in $\N^*\backslash J_1$ gives
$$
C_T(a_T^\star)\leq \gamma_1(T)\min_{\sum_{j\in J_1}|b_j|^2}\int_\Omega a(x)\bigg|\sum_{j\in J_1}b_j\phi_j(x)\bigg|^2\, dx=\gamma_1(T)\sigma_1(a_T^\star).
$$
Now, according to \eqref{mtgv:0818}, with the particular choice of parameter $\nu_T$ given by Lemma~\ref{lem:nuTdef}, we also have
$C_T(a_T^\star)\geq \gamma_1(T)\nu_T \sigma_1(a_1)$
if $T$ is large enough, and therefore
\begin{equation*}
\sigma_1(a_T^\star)\geq \frac{C_T(a_T^\star)}{\gamma_1(T)}\geq \nu_T\sigma_1(a_1).
\end{equation*}
In particular, this implies that 
$0\leq \sigma_1(a_1)-\sigma_1(a_T)\leq(1-\nu_T)\sigma_1(a_1)$.
Using the constant $K$ given by Theorem \ref{theo:sigma1} we deduce that
\[ K\Vert a_T^\star-a_1\Vert_{L^1(\O)}^2\leq (1-\nu_T)\sigma_1(a_1).\]  By construction of $\nu_T$ we have 
\[
\nu_T\sim e^{(\eta+2\Re(\lambda_1-\lambda_{p_0}))T}\quad \text{as }T\to+\infty
\] 
for some fixed $\eta\in (\Re(\lambda_{p_0}-\lambda_1), 2\Re(\lambda_{p_0}-\lambda_1))$, whence the conclusion.

\section{Conclusion and further comments}\label{sec:concl}

In this paper, we have described the large-time behavior of the maximizers of the observability constant associated with measurements of the solution of parabolic equations on a subdomain. 

Hereafter, we show by duality that our optimal observability results can be applied as well to the optimal actuator (controllability) shape and location problem. We then give a partial result regarding the small-time asymptotics, i.e., as $T\rightarrow 0$, of the optimal observability problem. We conclude with some open questions.

\subsection{Controllability}\label{sec:biblio}
The control system
\begin{equation}\label{heatEqcontrolled_intro}
\partial_ty+A_0y=u\mathds{1}_\omega 
\end{equation}
is said to be exactly null controllable in time $T$ if every initial datum $y(0,\cdot)\in L^2(\Omega)$ can be steered to $0$ in time $T$ by means of an appropriate control function $u\in L^2((0,T)\times \Omega)$.
It is well known that controllability and observability are dual notions (see \emph{e.g.} \cite{CurtainZwart, Trelat, TucsnakWeiss, Zabczyk}), in the sense that the latter exact null controllability property is equivalent to the observability property of \eqref{heatEq} on $\omega$ in time $T$, and the observability constant $C_T(\mathds{1}_\omega)$ coincides with the inverse of the minimal $L^2$ control cost for the controllability problem for \eqref{heatEqcontrolled_intro}.
The Hilbert Uniqueness Method (HUM, see \cite{MR953547, MR931277, Trelat, Zabczyk}) provides a characterization unique control solving the exact null controllability problem with a minimal $L^2$ norm. This control is referred to as the HUM control and is characterized as follows. Define the HUM functional $\mathcal{J}_\omega$ by
\begin{equation*}
\mathcal{J}_\omega(\phi^T)=\frac{1}{2}\int_0^T\!\!\!\int_\omega\vert\phi(t,x)\vert^2 \, dx \, dt + \langle\phi^T,y^0\rangle_{L^2},
\end{equation*}
where $\phi$ is the solution of $-\partial_t \phi+A_0^*\phi=0$ such that $\phi(T,\cdot)=\phi^T$.
Under the observability assumption, the functional $\mathcal{J}_\omega$ has a unique minimizer and the HUM control $u_\omega$ steering $y^0$ to $0$ in time $T$ is $u_\omega(t,x)=\mathds{1}_\omega(x)\phi(t,x)$. 
The HUM operator $\Gamma_\omega : L^2(\Omega) \rightarrow L^2((0,T)\times\Omega)$ is defined by $\Gamma_\omega(y^0) = u_\omega$ and its operator norm is given by
$$
\Vert \Gamma_\omega\Vert = \sup \left\{\frac{\Vert u_\omega\Vert_{L^2((0,T)\times\Omega)}}{\Vert y^0\Vert_{L^2(\Omega)}}\mid y^0 \in L^2(\Omega) \setminus\{0)\} \right\}.
$$
Adopting the convention $1/0=+\infty$, we have
$$
\Vert \Gamma_\omega\Vert = \frac{1}{C_T(\mathds{1}_\omega)}.
$$
We refer to \cite{MR3132418} for details.
Thus
$$
\inf_{\mathds{1}_\omega\in\mathcal{U}_L}\Vert \Gamma_\omega\Vert =\Big(\displaystyle \sup_{\mathds{1}_\omega\in\mathcal{U}_L}C_T(\mathds{1}_\omega)\Big)^{-1} .
$$
Hence, the problem of maximizing the observability constant is equivalent to the problem of minimizing the operator norm of $\Gamma_\omega$, which models the determination of a best control domain in some sense.

\subsection{Small-time asymptotics}
In this section, ,we mention an issue similar to the one investigated in this article, which cannot be treated using the same tools as in this paper: the limit of maximizers in small time, as $T\rightarrow 0$.
We have the following partial result. It would be relevant to study how to obtain quantitative estimates on the convergence of $\bar\delta_T/T$, as well as on the convergence of the maximizers.

\begin{theorem}\label{theo:smallT}
Assume that $q=1$ and that $A_0$ is the second-order elliptic differential operator
$$
A_0 = -\sum_{j=1}^d \partial_{x_j} (g_{jk} \partial_{x_k} ) + \sum_{j=1}^d b_j \partial_{x_j} + c
$$
where the functions $g_{jk}$, $b_j$ and $c$ are smooth on $\overline{\Omega}$. We assume that the matrix $G=(g_{jk})_{1\leq j,k\leq d}$ is real and positive definite in $\overline{\Omega}$, that $\det(G)=1$ on $\Omega$ and that $A_0$ is symmetric for the Lebesgue measure on $\Omega$.
Then 
$$
\bar\delta_T \sim LT\quad\textrm{as}\quad T\to 0 .
$$
\end{theorem}

\begin{proof}
Using \eqref{defCTdeterministic} and the fact that $(\phi_j)_{j\geq 1}$ is a Hilbert basis, we first note that, taking $a\equiv L$,
$$
C_T(L) = \inf_{\sum_{j=1}^{+\infty}|b_j|^2=1}\ \  \int_0^T\!  \int_\Omega L\bigg|\sum_{j=1}^{+\infty}b_je^{\lambda_j t}\phi_j(x)  \bigg|^2 dx \, dt= L\inf_{j\in\N^*}\gamma_j(T) = L\gamma_1(T).
$$
Since $a^\star_T$ is a maximizer, we have $\bar\delta_T = C_T(a^\star_T)\geq C_T(L)$ and thus
$$
\liminf_{T\to 0}\frac{\bar\delta_T}{T} \geq  L \liminf_{T\to 0} \frac{\gamma_1( T)}{T} = L.
$$
Let us prove the converse inequality. Actually, let us prove the stronger fact that
$$
\limsup_{T\to 0}\frac{C_T(a)}{T} \leq  L \qquad \forall a\in\overline{\mathcal{U}}_L .
$$
Using that $C_T(a)$ is defined by \eqref{defCTdeterministic} as an infimum, taking $b_j=1$ and all other $b_k$ equal to $0$, we have in particular
$$
C_T(a) \leq \int_0^T \int_\Omega a(x) e^{\Re\lambda_j t}\vert\phi_j(x)\vert^2\, dx\, dt = \gamma_j(T) \int_\Omega a\vert\phi_j\vert^2\qquad \forall j\in\N^*
$$
and thus, for every $N\in\N^*$, 
$$
C_T(a) \leq \inf \bigg\{ \sum_{j=1}^N \alpha_j \gamma_j(T) \int_\Omega a\vert\phi_j\vert^2 \ \ \Big\vert\ \   \alpha_1,\ldots,\alpha_N\geq 0, \ \sum_{j=1}^N\alpha_j=1 \bigg\} .
$$
Choosing in particular $\alpha_j = \gamma_j(T)^{-1} / \sum_{k=1}^N \gamma_k(T)^{-1}$ for every $j\in\{1,\ldots,N\}$, we infer that
\begin{equation}\label{ineq18:07}
\frac{1}{N}\sum_{j=1}^N \frac{1}{\gamma_j(T)} C_T(a) \leq \int_\Omega a \frac{1}{N}\sum_{j=1}^N \vert\phi_j\vert^2
\end{equation}
The inequality \eqref{ineq18:07} is valid for every $a\in\overline{\mathcal{U}}_L$, every $N\in\N^*$ and every $T>0$. Fixing $N$ and letting $T$ tend to $0$, using that $\gamma_j(T)\to T$, we obtain
$$
\limsup_{T\to 0}\frac{C_T(a)}{T} \leq \int_\Omega a \frac{1}{N}\sum_{j=1}^N \vert\phi_j\vert^2 \qquad\forall N\in\N^*.
$$
To conclude, we use the fact that the Ces\`aro mean $\frac{1}{N}\sum_{j=1}^N\vert\phi_j\vert^2$ converges to the constant function $\frac{1}{\vert\Omega\vert}$ as $N\rightarrow+\infty$, 
uniformly on any compact subset of the open set $\Omega$ for the $C^0$ topology and thus weakly in $L^2(\Omega)$ (this follows from \cite[Section 17.5, Theorem 17.5.7 and Corollary 17.5.8]{MR781536}).
The expected result follows. Note that a similar argument was used in \cite[Lemma 1]{MR3500831}.
\end{proof}

Theorem \ref{theo:smallT} gives the asymptotic behavior of the optimal value $\bar\delta_T$, but we do not know ho to study the behavior of the maximizers $a_T^\star$ as $T\to 0$, as we did in the case $T\to+\infty$, in particular due to the fact that we are not able to identify a limit problem as $T\to 0$. 
We let this question as an open problem. 

\subsection{Other open questions}
Other open questions are in order:
\begin{itemize}
\item[(i)] A first one is to understand whether optimal profiles $a_T^\star$ are characteristic functions for any $T>0$. This {\it bang-bang} property is usually a first step in deriving finer estimates and geometric information.
\item[(ii)] A second one is to establish quantitative estimates for the optimal observability constant: can we obtain an estimate akin to $C_T(a)-C_T(a_T^\star)\leq - K\Vert a-a_T^\star\Vert_{L^1(\O)}^2$? To answer this question, we believe that a positive answer should first be obtained regarding the bang-bang property.
\end{itemize}

\paragraph{Acknowledgment.}
We warmly thank Sylvain Ervedoza for his enlightening discussions on modal decompositions and several ideas of examples introduced in Section~\ref{sec:comments}.\\
The authors acknowledge the support of the ANR project TRECOS, grant number ANR-20-CE40-0009.  I. M-F was partially funded by a PSL Young Researcher Starting Grant 2023. 

\bibliographystyle{abbrv}
\bibliography{biblio_ObsHeat}

\end{document}